\documentclass[notitlepage,reqno,11pt]{amsart}
\RequirePackage[OT1]{fontenc}
\usepackage[foot]{amsaddr}
\RequirePackage{amssymb,nicefrac,bm,upgreek,mathtools,verbatim,enumerate,cite}
\RequirePackage[mathscr]{eucal}
\RequirePackage{dsfont}
\RequirePackage[normalem]{ulem}

\usepackage{scalerel}
\usepackage{tikz}
\usetikzlibrary{svg.path}

\definecolor{orcidlogocol}{HTML}{A6CE39}
\tikzset{orcidlogo/.pic={\fill[orcidlogocol]
svg{M256,128c0,70.7-57.3,128-128,128C57.3,256,0,198.7,0,128C0,57.3,57.3,0,128,
    0C198.7,0,256,57.3,256,128z};
\fill[white] svg{M86.3,186.2H70.9V79.1h15.4v48.4V186.2z}
svg{M108.9,79.1h41.6c39.6,0,57,28.3,57,53.6c0,27.5-21.5,53.6-56.8,
53.6h-41.8V79.1zM124.3,172.4h24.5c34.9,0,42.9-26.5,
42.9-39.7c0-21.5-13.7-39.7-43.7-39.7h-23.7V172.4z}
svg{M88.7,56.8c0,5.5-4.5,10.1-10.1,10.1c-5.6,0-10.1-4.6-10.1-10.1c0-5.6,
4.5-10.1,10.1-10.1C84.2,46.7,88.7,51.3,88.7,56.8z};}}

\newcommand\orcidicon[1]{\href{https://orcid.org/#1}{\mbox{\scalerel*{
\begin{tikzpicture}[yscale=-1,transform shape]
\pic{orcidlogo};
\end{tikzpicture}}{|}}}}

\usepackage[final]{hyperref}

\usepackage[margin=1in]{geometry}
\numberwithin{equation}{section}

\theoremstyle{plain}
\newtheorem{theorem}{Theorem}[section]

\theoremstyle{definition}
\newtheorem{assumption}{Assumption}[section]
\newtheorem{definition}{Definition}[section]

\theoremstyle{remark}
\newtheorem{remark}{Remark}[section]

\newcommand{\stkout}[1]{\ifmmode\text{\sout{\ensuremath{#1}}}\else\sout{#1}\fi}
\usepackage{color}
\definecolor{mred}{rgb}{.7,.0,.0}
\definecolor{dred}{rgb}{.6,.0,.0}
\definecolor{mblue}{rgb}{.0,.0,.8}
\definecolor{dblue}{rgb}{.0,.0,.5}
\definecolor{mgreen}{rgb}{.0,.6,.0}
\definecolor{dgreen}{rgb}{.0,.4,.0}
\definecolor{dmagenta}{rgb}{.4,.1,.5}

\hypersetup{
 colorlinks=true,
 citecolor=mblue,
 linkcolor=mblue,
 frenchlinks=false,
 pdfborder={0 0 0},
 naturalnames=false,
 hypertexnames=false,
 breaklinks}
\usepackage[capitalize,nameinlink]{cleveref}

\crefname{section}{Section}{Sections}
\crefname{subsection}{subsection}{subsections}
\crefname{notation}{Notation}{Notations}
\crefname{hypothesis}{Hypothesis}{Conditions}
\crefname{assumption}{Assumption}{Assumptions}

\Crefname{figure}{Figure}{Figures}

\crefformat{equation}{\textup{#2(#1)#3}}
\crefrangeformat{equation}{\textup{#3(#1)#4--#5(#2)#6}}
\crefmultiformat{equation}{\textup{#2(#1)#3}}{ and \textup{#2(#1)#3}}
{, \textup{#2(#1)#3}}{, and \textup{#2(#1)#3}}
\crefrangemultiformat{equation}{\textup{#3(#1)#4--#5(#2)#6}}%
{ and \textup{#3(#1)#4--#5(#2)#6}}{, \textup{#3(#1)#4--#5(#2)#6}}%
{, and \textup{#3(#1)#4--#5(#2)#6}}

\Crefformat{equation}{#2Equation~\textup{(#1)}#3}
\Crefrangeformat{equation}{Equations~\textup{#3(#1)#4--#5(#2)#6}}
\Crefmultiformat{equation}{Equations~\textup{#2(#1)#3}}{ and \textup{#2(#1)#3}}
{, \textup{#2(#1)#3}}{, and \textup{#2(#1)#3}}
\Crefrangemultiformat{equation}{Equations~\textup{#3(#1)#4--#5(#2)#6}}%
{ and \textup{#3(#1)#4--#5(#2)#6}}{, \textup{#3(#1)#4--#5(#2)#6}}%
{, and \textup{#3(#1)#4--#5(#2)#6}}

\crefdefaultlabelformat{#2\textup{#1}#3}

\newcommand{\grad}{\nabla}

\newcommand{\process}[1]{{\{#1_t\}_{t\ge0}}}
\newcommand{\ttup}[1]{\textup{(}#1\textup{)}}
\newcommand{\Ind}{\mathds{1}}           
\newcommand{\transp}{^{\mathsf{T}}}     

\newcommand{\Exp}{{\mathbb{E}}}         
\newcommand{\Prob}{{\mathbb{P}}}        
\newcommand{\RR}{{\mathbb R}}           
\newcommand{\Rd}{{{\mathbb R}^d}}
\newcommand{\NN}{{\mathbb N}}           
\newcommand{\D}{\mathrm{d}}             
\newcommand{\E}{\mathrm{e}}             
\newcommand{\df}{\coloneqq}             

\newcommand{\sB}{{\mathscr{B}}}         
\newcommand{\Act}{\mathbb{U}}           
\newcommand{\Uadm}{\mathfrak{U}}        
\newcommand{\Usm}{\mathfrak{U}_{\mathsf{sm}}}  

\newcommand{\uuptau}{\Hat{\uptau}}

\newcommand{\Sobl}{\mathscr{W}_{\mathrm{loc}}}  

\newcommand{\Equil}{\mathscr{E}}        
\newcommand{\sF}{\mathfrak{F}}          
\newcommand{\cG}{{\mathcal{G}}}        
\newcommand{\cH}{{\mathcal{H}}}         

\newcommand{\cK}{{\mathcal{K}}}         
\newcommand{\Lg}{\mathcal{L}}           
\newcommand{\cP}{{\mathcal{P}}}         
\newcommand{\cS}{\mathcal{S}}           
\newcommand{\cU}{\mathcal{U}}           
\newcommand{\Lyap}{\mathscr{V}}         

\newcommand{\tv}{{\rule[-.45\baselineskip]{0pt}{\baselineskip}\mathsf{TV}}}

\newcommand{\abs}[1]{\lvert#1\rvert}
\newcommand{\norm}[1]{\lVert#1\rVert}
\newcommand{\babs}[1]{\bigl\lvert#1\bigr\rvert}

\newcommand{\bnorm}[1]{\bigl\lVert#1\bigr\rVert}

\DeclareMathOperator{\trace}{Tr}

\newcommand{\ttl}{\Large \bf On the relative value iteration
with a risk-sensitive criterion}

\begin{document}
\title[Risk-sensitive Relative Value Iteration]{\ttl}

\author[Ari Arapostathis]{Ari Arapostathis$^\dag$}
\address{$^\dag$Department of Electrical and Computer Engineering,
The University of Texas at Austin, 2501 Speedway,
EER~7.824,
Austin, TX~78712}
\email{ari@utexas.edu\,\protect\orcidicon{0000-0003-2207-357X}}

\author[Vivek S. Borkar]{Vivek S. Borkar$^\ddag$}
\address{$^\ddag$Department of Electrical Engineering,
Indian Institute of Technology, Powai, Mumbai 400076, India}
\email{borkar@ee.iitb.ac.in\,\protect\orcidicon{0000-0003-0756-5402}}

\begin{abstract}
A multiplicative relative value iteration algorithm for solving the dynamic
programming equation for the risk-sensitive control problem is studied for discrete
time controlled Markov chains with a compact Polish state space,
and controlled diffusions in on the whole Euclidean space.
The main result is a proof of convergence to the desired limit in each case.
\end{abstract}



\maketitle

\section{Introduction}

Risk-sensitive control problems on an infinite horizon
seek to minimize or maximize a functional defined
as the exponential growth rate of a multiplicative cost, resp.\ reward.
Thus unlike the more classical and commonplace criteria, they lead to a multiplicative
dynamic programming equation, in fact a nonlinear eigenvalue problem for a positive,
positively 1-homogeneous continuous nonlinear operator. This has been extensively
studied for the discrete time discrete state (both finite and countable) and
continuous time and state problems, but the important case of discrete time and
general state space has received relatively less attention in comparison, with
only a small number of contributions such as
\cite{Ananth,DiMasi,AnnaJ}.
The same also holds for the corresponding development of the value iteration algorithm,
which ends up being a multiplicative analog of the algorithm encountered in
average cost problems, alternatively, in its simplest scenario,
a nonlinear counterpart of the power iteration
method for computing the principal eigenvector and eigenvalue of an irreducible
non-negative matrix. This again has been studied in the discrete time and state case
\cite{BorMeyn-02,Cadena1,Cadena2}, but not for the general state space.
In this work we take a first step towards filling in this gap by proposing and analyzing
a multiplicative relative value iteration algorithm for two instances of risk-sensitive
control on a general state space: the discrete time compact Polish state space problem,
and the continuous time controlled diffusion in a Euclidean space.
In the case of controlled diffusions, we would like to cite here the work in
\cite{Ichihara-13,Ichihara-15,Ichihara-19,Kaise-06} which is very much related to
this problem.

\section{Results in Discrete Time}

We consider a controlled Markov chain on a compact Polish space $\cS$ 
with a compact metric action space $U$ and controlled transition kernel
$$(x, u) \in \cS\times U \mapsto p(\D{y} \,|\, x, u) \,=\,
\varphi(y \,|\, x,u)\gamma(\D{y}) \in \cP(\cS)\,,$$
where $\gamma$ is some positive measure on $\cS$ with full support and
$\varphi( \cdot \,|\, \cdot , \cdot ) > 0$ is continuous.
Also given is a `per stage' continuous cost function
$$(x, u) \in \cS\times U \mapsto k(x, u)\,.$$
We shall denote by $X_n, n \ge 0$, and $Z_n$, $n \ge 0$, resp.,
the $\cS$-valued state process and $U$-valued control process.
Thus
$$P(X_{n+1} \in A \,|\, X_m, Z_m, \ m \le n) \,=\,
p(A \,|\, X_n, Z_n) \qquad \forall\, n\in\NN\,,\ \forall\, A \ \mbox{Borel in} \ \cS\,.$$
When $Z_n = v(X_n)$ for all $n$ for some measurable $v \colon \cS \mapsto U$,
we call it a stationary Markov control policy and denote is simply by $v$.
When
$$P(Z_n\in B \,|\, X_m, Z_m, m < n; X_n) \,=\,\phi(B \,|\, X_n)\qquad\forall\,n\in\NN\,,$$
for some $\phi \colon \cS \mapsto \cP(U)$, we call it a randomized Markov control policy
and denote it simply by $\phi$.

The objective is to minimize the asymptotic risk-sensitive cost
\begin{equation*}
\limsup_{n\uparrow\infty}\,
\frac{1}{n}\,
\log \Exp\left[\E^{\sum_{m=0}^{n-1}k(X_m, Z_m)}\right]\,. 
\end{equation*}
The `dynamic programming equation' for this problem ends up being the nonlinear
eigenvalue problem
\begin{equation}\label{DP}
\Lambda V(x) \,=\,\min_{u\in U}\,\left(\E^{k(x,u)}\int_\cS
p(\D{y} \,|\, x,u)V(y)\right)\,,\quad x \in \cS\,. 
\end{equation}
By Theorem 2.2 of \cite{Ananth}, this  has a solution
$V(\cdot) \in C\bigl(\cS; [0,\infty)\bigr)$, $\Lambda \in (0, \infty)$,
where $\Lambda$ is unique, and $V$ is unique up to a multiplicative positive scalar.
Our objective is to propose a recursive scheme to compute these.
Specifically, we consider the `Value Iteration' (VI) algorithm given by
\begin{equation*}
\begin{aligned}
J_{n+1}(x) &\,=\, \frac{\min_{u\in U}\int_\cS p(\D{y} \,|\, x,u)\,
\E^{k(x,u)}J_n(y)}{\Lambda} \\ 
&\,=\, \frac{\int_\cS p(\D{y} \,|\, x,u_n(x))\,\E^{k(x,u)}J_n(y)}{\Lambda}
\end{aligned}
\end{equation*}
for suitable $u_n(\cdot)$ guaranteed by a standard measurable selection theorem \cite{TP}.
 This is not a practicable algorithm since $\Lambda$ is unknown.
 But it will serve a useful purpose in the analysis of the more realistic scheme,
the `Relative Value Iteration' (RVI).
Choose some  $x_0\in\cS$, which is kept fixed.
 The RVI is given by
\begin{equation*}
\begin{aligned}
V_{n+1}(x) &\,=\, \frac{\min_{u\in U}\int_\cS p(\D{y} \,|\, x,u)\,
\E^{k(x,u)}V_n(y)}{V_n(x_0)} \\ 
 &\,=\, \frac{\int_\cS p(\D{y} \,|\, x,u_n''(x))\,\E^{k(x,u)}V_n(y)}{V_n(x_0)}\,,
\end{aligned}
\end{equation*}
for suitable $u_n''(\cdot)$, initiated at $J_0 = V_0 > 0$ so that $V_n, J_n > 0$ 
for all $n$.

We have
\begin{align*}
\max_{x\in\cS}\left(\frac{V_{n+1}(x)}{J_{n+1}(x)}\right)
&\,=\, \max_{x\in\cS}\left(\frac{\min_{u\in U}\int_S
p(\D{y} \,|\, x,u)\,\E^{k(x,u)}J_n(y)\left(\frac{V_n(y)}{J_n(y)}
\right)}{\min_{u\in U}\int_S p(\D{y} \,|\, x,u)\,\E^{k(x,u)}J_n(y)}\right)
\frac{\Lambda}{V_n(x_0)} \\
&\,\le\, \max_{x\in\cS}\left(\frac{V_n(x)}{J_n(x)}\right)\frac{\Lambda}{V_n(x_0)}\,.
\end{align*}
Similarly,
\begin{equation*}
\min_{x\in\cS}\left(\frac{V_{n+1}(x)}{S_{n+1}(x)}\right)
\,\ge\, \min_{x\in\cS}\left(\frac{V_n(x)}{J_n(x)}\right)\frac{\Lambda}{V_n(x_0)}\,.
\end{equation*}
Therefore
\begin{equation*}
1 \,\le\, \frac{\max_{x\in\cS}\,\Bigl(\frac{V_{n+1}(x)}{J_{n+1}(x)}\Bigr)}
{\min_{x\in\cS}\,\Bigl(\frac{V_{n+1}(x)}{J_{n+1}(x)}\Bigr)} \,\le\, 
\frac{\max_{x\in\cS}\,\Bigl(\frac{V_{n}(x)}{J_{n}(x)}\Bigr)}
{\min_{x\in\cS}\,\Bigl(\frac{V_{n}(x)}{J_{n}(x)}\Bigr)}
\,\le\, \dotsb \,\le\, 1\,,
\end{equation*}
implying that equality must hold throughout, that is, $V_n(x) = C_nJ_n(x)$
for some constant $C_n$ independent  of $x$. 
We can then show inductively that
\begin{equation*}
C_n \,\df\, \frac{V_{n}(x)}{J_{n}(x)} \,=\,
\prod_{m=0}^{n-1}\frac{\Lambda}{V_m(x_0)}\,.
\end{equation*}
Furthermore,
\begin{equation*}
\frac{V_{n+1}(x_0)}{J_{n+1}(x_0)} \,=\,
\frac{V_{n}(x_0)}{J_{n}(x_0)}\frac{\Lambda}{V_n(x_0)} \,=\,\frac{\Lambda}{J_n(x_0)}\,.
\end{equation*}
We say that the VI (RVI) converges if the sequence of
functions $\{J_n\}_{n\in\NN}$ ($\{V_n\}_{n\in\NN}$) converges pointwise.
If the VI converges, in particular $J_n(x_0)$ does, and by the above equations,
the RVI will also converge. Thus we only need to establish the convergence of the VI.

Let $V(\cdot)$ and $\Lambda$ be as in \cref{DP}.
Let $v^{*}(\cdot)$ denote a measurable minimizer of the right hand side of \cref{DP}.
This is always possible by a measurable selection theorem \cite{TP}. Define
$$p^*(\D{y} \,|\, x) \,\df\,
\bigl(\Lambda V^*(x)\bigr)^{-1}
p\bigl(\D{y} \,|\, x,v^*(x)\bigr)\,\E^{k(x,v^*(x))}V^*(y)\,.$$ 
Then we have,
$$\frac{J_{n+1}(x)}{V^*(x)} \,\le\, \int_\cS p^*(\D{y} \,|\, x)
\left(\frac{J_n(y)}{V^*(y)}\right)\,.$$
Let $\{X^*_n\}$ denote the stationary
chain governed by $p^*(\cdot  \,|\,  \cdot)$.
Let $Y_n\df X^*_{-n}$, for $n\in\NN$.
It then follows that
$$\frac{J_n(Y_n)}{V^*(Y_n)}\,,\quad n < 0\,,$$
is a  reverse submartingale that converges a.s.\ and in $L_1(\nu)$ \cite{Neveu}
to a random variable $\zeta$ (say).
For any open $O \subset \cS$, the martingale law of large numbers \cite{Neveu} yields
$$\lim_{n\uparrow\infty}\,\frac{1}{n}\,\sum_{m=0}^{n-1}\Bigl(I\{X^*_{m+1} \in O\}
- p^*\bigl(O \,|\,  X^*_m, v^*(X^*_m)\bigr)\Bigr) \,=\,0 \text{\ \ a.s.}$$
Under our assumptions, 
\begin{equation}\label{uniformbound}
\min_{x,u}\, p^*(O \,|\, x,u) \,>\, \delta \gamma (O) \,>\, 0 
\end{equation}
 for some $\delta > 0$.
 Thus
\begin{equation*}
\liminf_{n\uparrow\infty}\, \frac{1}{n}\, \sum_{m=0}^{n-1}\Ind\{X^*_n \in O\}
\,\ge\, \delta \gamma(O)\text{\ \ a.s.},
\end{equation*}
implying $X^*_n \in O$ i.o., a.s.
Fix $\eta > 0$ and let $O$ be an open $\epsilon$-ball centered at $x$ for a
prescribed $\epsilon > 0$, chosen such that
\begin{equation*}
y \in O \Longrightarrow \abs{V(y) - V(x)} \,<\, \eta\,.
\end{equation*}
Pick a zero probability set $\mathcal{N}$ outside which all `a.s.' results above hold for
$\epsilon = \frac{1}{m}$, $\eta = \frac{1}{k}$, and $m, k \ge 1$.
Fix $x \in \cS$. Fix a sample point $\omega \notin \mathcal{N}$.
Take (possibly random) $n_0 \ge 1$ such that (say)
\begin{equation*}
n \ge n_0 \,\Longrightarrow \,\abs{J_n(X^*_n) - \zeta V(X^*_n)} \,<\, \eta
\,=\, \frac{1}{k}\,.
\end{equation*}
 Then on $\{X^*_n \in O\}$ with $\epsilon = \frac{1}{m}$ (say), we have
\begin{equation*}
\begin{aligned}
\babs{J_n(X^*_n) - \zeta V^*(x)} &\,\le\,
\babs{J_n(X^*_n) - \zeta V^*(X^*_n)}
+ \babs{\zeta (V^*(X^*_n) - V^*(x))}\\
&\,\le\, (\zeta + 1)\frac{1}{k}\,.
\end{aligned}
\end{equation*}
Considering $k, m \uparrow \infty$, it follows that if $X^*_n \to x$ along a subsequence,
 then $J_n(X^*_n) \to \zeta V^*(x)$ along that subsequence.
 By \cref{uniformbound}, it also follows that $J_n(x) \to \zeta V^*(x)$ for
 $\gamma$-a.s.\ $x$. It then follows that $V_n(x) \to$ some $\Bar{V}(x)$ $\gamma$-a.s.
 But then, passing to the limit in the defining equation for RVI, $\Bar{V}$
 satisfies  \cref{DP} with $\Bar{V}(x_0)=\Lambda$, which uniquely specifies it.

\section{Results in Continuous Time}

In this section we consider the risk-sensitive control problem for
a controlled diffusion on $\Rd$ taking the form
\begin{equation}\label{E-sde}
\D X_t \,=\, b(X_t,U_t)\,\D t + \upsigma (X_t)\,\D W_t\,.
\end{equation}
All random processes in \cref{E-sde} live in a complete
probability space $(\Omega,\sF,\Prob)$.
The process $W$ is a $d$-dimensional standard Wiener process independent
of the initial condition $X_{0}$, and
the control process $\{U_t\}_{t\ge0}$ lives in a compact metrizable space $\Act$.
The sets of admissible controls $\Uadm$, and stationary
Markov controls $\Usm$ are defined in the standard manner.

We let $a\df \upsigma\upsigma\transp$,
and denote by $B_R$ the open ball of radius $R$ in $\Rd$
centered at $0$.
We impose the following set assumptions on the coefficients,
and the running cost $c\colon\Rd\times\Act\to\RR$.

\begin{assumption}\label{A3.1}
The following hold.
\begin{enumerate}
\item[(i)]
The drift $b\colon\RR^{d}\times\Act\to\RR^{d}$ and running
cost $c$ are continuous,
and for some positive constants $C_R$ depending on $R>0$, and $C_0$, we have
\begin{gather*}
\abs{c(x,u) - c(y,u)} +
\abs{b(x,u) - b(y,u)} + \norm{\upsigma(x) - \upsigma(y)} \,\le\,C_{R}\,\abs{x-y}
\intertext{for all $x,y\in B_R$ and $u\in\Act$, and}
\sum_{i,j=1}^{d} a^{ij}(x)\zeta_{i}\zeta_{j}
\,\ge\, C^{-1}_0 \abs{\zeta}^{2}
\quad\forall\, (x,\zeta)\in \Rd\times\Rd\,,
\end{gather*}
where $\norm{\upsigma}\df\bigl(\trace\, \upsigma\upsigma\transp\bigr)^{\nicefrac{1}{2}}$
denotes the Hilbert--Schmidt norm of the matrix $\upsigma$.

\smallskip
\item[(ii)]
The function $a\colon \Rd\to\RR^{d\times d}$ is bounded,
and for some $\theta\in[0,1)$
and a constant $\kappa_0$, we have
\begin{equation}\label{EA3.1B}
\abs{b(x,u)} \,\le\, \kappa_0\bigl(1+\abs{x}^\theta\bigr)\,,\quad\text{and\ \ }
\abs{c(x,u)} \,\le\, \kappa_0\bigl(1+\abs{x}^{2\theta}\bigr)
\end{equation}
for all $(x,u)\in\Rd\times\Act$.
In addition,
\begin{equation}\label{EA3.1C}
\min_{x\in B_R}\,\min_{u\in\Act}\,c(x,u)
\;\xrightarrow[R\to\infty]{}\;\infty\,,
\end{equation}
and
\begin{equation}\label{EA3.1D}
\max_{x\in B_R}\,\frac{1}{\abs{x}^{1-\theta}}\;
\max_{u\in\Act}\;\bigl\langle b(x,u),\, x\bigr\rangle^{+}
\;\xrightarrow[R\to\infty]{}\;0\,.
\end{equation}
\end{enumerate}
\end{assumption}

\begin{definition}
For $U\in\Uadm$ we define \emph{the risk-sensitive value} under
a control $U\in\Usm$, by
\begin{equation}\label{E-valueU}
\Lambda^x_U\,=\, \Lambda^x_U(c)\,\df\, \limsup_{T\to\infty}\,\frac{1}{T}\,
\log\Exp^x_U\Bigl[\E^{\int_{0}^{T} c(X_{t},U_t)\,\D{t}}\Bigr]\,,
\end{equation}
and the \emph{risk-sensitive optimal values} by
\begin{equation}\label{E-value}
\Lambda^x_* \,\df\, \inf_{U\in\,\Uadm}\,\Lambda^x_U\,,\quad\text{and}\quad
\Lambda_* \,\df\, \inf_{x\in\,\Rd}\,\Lambda^x_*\,.
\end{equation}
Also let
\begin{equation*}
\cG f(x) \,\df\, \frac{1}{2}\trace\left(a(x)\nabla^{2}f(x)\right)
+ \min_{u\in\Act}\, \bigl[\bigl\langle b(x,u),
\grad f(x)\bigr\rangle + c(x,u) f(x)\bigr]\,,\quad f\in C^2(\Rd)\,,
\end{equation*}
and
\begin{equation}\label{E-lamstr}
\lambda_*\,=\,\lambda_*(c)\,\df\,\inf\,\Bigl\{\lambda\in\RR\,
\colon \exists\, \phi\in\Sobl^{2,d}(\Rd),\ \phi>0, \
\cG\phi -\lambda\phi\le 0 \text{\ a.e.\ in\ } \Rd\Bigr\}\,.
\end{equation}
\end{definition}

Some discussion is in order here.
The quantity $\lambda_*$ is the generalized principal eigenvalue of the
semilinear operator $\cG$ in $\Rd$.
We assume that $\lambda_*<\infty$.
Note that in specific problems, this is verified via a Foster--Lyapunov
equation of the form
\begin{equation*}
\frac{1}{2}\trace\left(a(x)\nabla^{2}\Lyap(x)\right)
+ \bigl\langle b_v(x),\grad \Lyap(x)\bigr\rangle
+ c_v(x) \Lyap(x)\,\le\, \kappa_0 -\kappa_1 \Lyap(x)
\end{equation*}
for some positive function $\Lyap\in C^2(\Rd)$ which is bounded
away from $0$, and for some $v\in\Usm$ and constants $\kappa_0$
and $\kappa_1$.
In this equation we used the convenient notation
\begin{equation*}
b_v(x) \,\df\, b\bigl(x,v(x)\bigr)\,,\quad\text{and\ }
c_v(x) \,\df\, c\bigl(x,v(x)\bigr)\qquad\text{for\ }v\in\Usm\,,
\end{equation*}
which we adopt for the rest of the paper.

\subsection{The risk-sensitive HJB}\label{S3.1}

As shown in \cite[Lemmas~2.2 and 2.3]{ABis-18}, there exists a positive
eigenfunction $\Psi\in C^2(\Rd)$ which solves
\begin{equation}\label{E-eigen}
\cG \Psi(x) \,=\, \lambda_* \Psi(x)\,,\quad x\in\Rd\,,
\end{equation}
and $\lambda_*\le \Lambda^x_*$ for all $x\in\Rd$.
We let $\Usm^*$ denote the controls $v\in\Usm$ which satisfy
\begin{equation*}
\bigl\langle b_v(x),\grad \Psi(x)\bigr\rangle + c_v(x) \Psi(x)
\,=\, \min_{u\in\Act}\bigl[
\bigl\langle b(x,u),\grad \Psi(x)\bigr\rangle + c(x,u) \Psi(x)\bigr]
\quad \text{a.e.\ }x\in\Rd\,.
\end{equation*}
In other words, $\Usm^*$ is the set of measurable selectors from
the minimizer of \cref{E-eigen}.

A variation of \cite[Lemma~3.2]{ABis-18}, using
\cref{EA3.1D}, shows that
\begin{equation}\label{E-vanish}
\limsup_{t\to\infty}\, \frac{1}{t}\,\Exp^x_U\bigl[\abs{X_{t}}^{1+\theta}\bigr]
\, =\, 0 \qquad\forall\,U\in\Uadm\,.
\end{equation}
Indeed, using the function $\abs{x}^{2(1+\theta)}$ in equation (3.1) of
\cite{ABis-18} following the rest of the proof of \cite[Lemma~3.2]{ABis-18},
we obtain \cref{E-vanish}.
On the other hand, \cite[Lemma~4.1]{ABBK-19} shows that
\cref{EA3.1B,EA3.1D} imply that there exists
a constant $\widetilde{C}_0>0$ such that any positive solution
$\phi\in\Sobl^{2,d}(\Rd)$ of
\begin{equation*}
\frac{1}{2}\trace\left(a(x)\nabla^{2}\phi(x)\right)
+ \bigl\langle b_v(x),
\grad \phi(x)\bigr\rangle + c_v(x) \phi(x)\,=\,\lambda\phi(x)
\end{equation*}
for $v\in\Usm$, satisfies
\begin{equation}\label{E-gradient}
\frac{\abs{\grad\phi(x)}}{\phi(x)} \,\le\,
\widetilde{C}_0 (1+\abs{x}^{\theta})
\end{equation}
Therefore, by \cref{E-gradient},
the eigenfunction $\Psi$ in \cref{E-eigen} satisfies
\begin{equation}\label{E-growth1}
\E^{-C(1+\abs{x}^{1+\theta})}\,\le\, \Psi(x) \,\le\,
\E^{C(1+\abs{x}^{1+\theta})}\qquad\forall\,x\in\Rd\,,
\end{equation}
for some constant $C>0$.
An application of Fatou's lemma on the stochastic representation of
the solution $\Psi$ of \cref{E-eigen} shows that
\begin{equation}\label{E-Fatou1}
\Psi(x) \,\ge\, \Exp^{x}_{v^*}
\Bigl[\E^{\int_{0}^{T} [c_{v^*}(X_{t})-\lambda_*]\,\D{t}}\,
\Psi(X_{T})\Bigr] \qquad \forall\,T>0\,,
\end{equation}
with $v^*\in\Usm^*$.
Taking logarithms on both sides of \cref{E-Fatou1},
applying Jensen's inequality, and
dividing by $T$, we obtain
\begin{equation}\label{E-Jensen1}
\frac{1}{T}\;\Exp^{x}_{v^*}\biggl[\int_{0}^{T}
c_{v^*}(X_{t})\,\D{t}\biggr]
+\frac{1}{T}\;\Exp^{x}_{v^*}\bigl[\log{\Psi(X_{T})}\bigr] \,\le\,
\lambda_* + \frac{1}{T} \log \Psi(x)\,.
\end{equation}
Using \cref{E-vanish,E-growth1} and taking limits as $T\to\infty$ in \cref{E-Jensen1},
 we obtain
\begin{equation*}
\limsup_{T\to\infty}\;
\frac{1}{T}\;\Exp^{x}_{v^*}\biggl[\int_{0}^{T}
c_{v^*}(X_{t})\,\D{t}\biggr]\,\le\,\lambda_*\,.
\end{equation*}
This together with \cref{EA3.1C} implies that the diffusion in \cref{E-sde}
controlled by $v^*\in\Usm^*$ has an invariant probability measure,
and, therefore, it is positive recurrent \cite[Theorem~3.3]{Hasm-60}
(see also \cite{Bhattacharya-78}).
An application of \cite[Lemma~2.1]{ABis-18} then shows that $\Psi$ is inf-compact,
which in turn implies that
$\Lambda^x_{v^*}\le\lambda_*$ for all $x\in\Rd$,
by \cite[Lemma~2.1\,(d) and (f)]{ABis-18}.
Since we have already asserted the converse inequality, this shows
that
\begin{equation*}
\Lambda^x_*\,=\,\Lambda_*\,=\,\lambda_*\qquad\forall\,x\in\Rd\,,
\end{equation*}
or in other words, the optimal risk-sensitive
value is equal to the generalized principal eigenvalue defined in \cref{E-lamstr}.
Note also that the inf-compactness of
$\Psi$ implies by \cref{E-eigen} that the diffusion in \cref{E-sde}
controlled under $v^*\in\Usm^*$ is exponentially ergodic, or in other words,
the transition probability of the process $\process{X}$ in \cref{E-sde} under
the control $v^*$,
converges to its invariant probability measure in total variation at
an exponential rate \cite{MT-III-93}.

Uniqueness of the eigenfunction $\Psi$, which we refer to as the \emph{ground state},
is related to the ergodic properties of the \emph{ground state diffusion},
which takes the form
\begin{equation}\label{E-sde*}
\D{X}^*_{t} \,=\, \bigl(b(X^*_t,U_t) + a (X^*_t)
\grad\psi(X^*_t)\bigr)\,\D{t} + \upsigma(X^*_t)\,\D W^*_{t}\,,
\end{equation}
with $\psi\df\log\Psi$.
First, we have equality in \cref{E-Fatou1} if and only if
\cref{E-sde*} controlled under $U_t=v^*(X^*_t)$ is regular.
This is shown in \cite[Lemma~2.3 and Corollary~2.2]{ABS-19}.
Note that \cref{E-eigen} can be written in the form 
\begin{equation}\label{E-cG*}
\begin{aligned}
\cG^* \psi(x)&\,\df\,
\frac{1}{2}\trace\left(a(x)\nabla^{2}\psi(x)\right)\\
&\mspace{50mu}
+ \min_{u\in\Act}\,\Bigl[\bigl\langle b(x,u)+ \tfrac{1}{2} a(x)\grad\psi(x),
\grad \psi(x)\bigr\rangle + c(x,u) \psi(x)\Bigr]\,=\,\lambda\psi(x)\,.
\end{aligned}
\end{equation}
Naturally, the sets of measurable selectors from the minimizers of
\cref{E-eigen} and \cref{E-cG*} are equal.
By \cref{E-gradient}, the hypothesis that $a$ is bounded,
and the growth assumptions of the drift in \cref{EA3.1B}, it follows
that \cref{E-sde*} is regular for any $U\in\Uadm$.
Thus, mimicking the proof of \cite[Lemma~2.3]{ABS-19} we obtain
\begin{equation}\label{E-crucial}
\Psi(x) \,\le\, \Exp^{\mathstrut x}_{ U}
\Bigl[\E^{\int_{0}^{T} [c(X_{t},U_t)-\lambda_*]\,\D{t}}\,
\Psi(X_{T})\Bigr] \qquad \forall\,T>0\,,
\end{equation}
with equality when $U_t=v^*(X_t)$ for any $v^*\in\Usm^*$.

We review one important property of the generalized principal eigenvalue
which concerns its dependence on the running cost $c$.
Let
\begin{equation}\label{E-Lg}
\Lg_u f(x) \,\df\,
\frac{1}{2}\trace\left(a(x)\nabla^{2}f(x)\right) +
\bigl\langle b(x,u),\grad f(x)\bigr\rangle\,,\qquad u\in\Act\,,
\end{equation}
and $\Lg_v$ for $v\in\Usm$, denote the operator defined as
above, but with $b(x,u)$ replaced by $b_v(x)$.
For $v\in\Usm$ let
\begin{equation}\label{E-lamstrv}
\lambda_v(c)\,\df\,\inf\,\Bigl\{\lambda\in\RR\,
\colon \exists\, \phi\in\Sobl^{2,d}(\Rd),\ \phi>0, \
\Lg_{v}\phi +c_{v}\phi-\lambda\phi\le 0 \text{\ a.e.\ in\ } \Rd\Bigr\}\,.
\end{equation}
Naturally, we have $\lambda_{v^*}(c)=\lambda_*$ for all $v^*\in\Usm^*$.
Let $C_{\mathrm{o}}^+(\Rd)$ denote the collection of all non-trivial,
nonnegative, continuous
functions which vanish at infinity.
We say that $\lambda_v$ is
\emph{strictly monotone at $c$ on the right}
if $\lambda_v(c+h)>\lambda_v(c)$ for all $h\in C_{\mathrm{o}}^+(\Rd)$.
We can of course define the analogous property for $\lambda_*$, independently
of the control $v^*\in\Usm^*$, using the definition in \cref{E-lamstr}.
Since $\Usm^*$ is the set of measurable selectors from the minimizer,
it is clear that these two properties are equivalent.

Let $\uuptau(A)$ denote the first hitting time of the set $A$.
By \cite[Lemma~2.7, Corollary~2.3, and Theorem~2.3]{ABS-19}, together
with the equivalence of strict monotonicity on the right of
$\lambda_*$ and $\lambda_{v^*}$ for $v^*\in\Usm^*$, we can assert that the following
statements are equivalent.
\begin{enumerate}
\item
The eigenvalue $\lambda_*$ is simple.
\smallskip\item
It holds that
\begin{equation}\label{E-strep}
\Psi(x)\,=\,\Exp^x_{v^*}
\Bigl[\E^{\int_0^{\uuptau(\sB)}[c_{v^*}(X_s)-\lambda_*]\, \D{s}}\,\Psi(X_{\uuptau(\sB)})
\,\Ind_{\{\uuptau(\sB)<\infty\}}\Bigr]\,,
\quad \forall\,x\in \Bar\sB^c\,,
\end{equation}
for any open ball $\sB$ and $v^*\in\Usm^*$.
\smallskip\item
The ground state process in \cref{E-sde*} controlled under
any $v^*\in\Usm^*$ is recurrent.
\end{enumerate}

We summarize the above discussion in the following theorem which is
a slight variation of \cite[Proposition~5.1]{ABBK-19}.

\begin{theorem}\label{T3.1}
Grant \cref{A3.1}, and suppose that $\lambda_*$ is finite.
Then the HJB equation
\begin{equation}\label{T3.1A}
\min_{u\in\Act}\;
\bigl[\Lg_u \Psi(x) + c(x,u)\,\Psi(x)\bigr] \,=\,  \lambda_*\,\Psi(x)
\qquad\forall\,x\in\Rd
\end{equation}
has a solution $\Psi\in C^2(\RR^{d})$,
satisfying $\inf_{\Rd}\,\Psi>0$, and the following hold:
\begin{enumerate}
\item[\ttup a]
$\Lambda^x_*=\Lambda_*=\lambda_*$ for all $x\in\Rd$.

\smallskip
\item[\ttup b]
Any $v^{\mathstrut*}\in\Usm^*$ renders the SDE in \cref{E-sde} exponentially ergodic
 and is optimal, that is, $\Lambda^x_{v^*}=\Lambda_*$ for all $x\in\Rd$.

\smallskip
\item[\ttup c]
It holds that
\begin{equation*}
\Psi(x) \,=\, \Exp^x_{v^*}\Bigl[\E^{\int_{0}^{T}
[c(X_{t},v(X_{t}))-\lambda_*]\,\D{t}}\,\Psi(X_T)\Bigr]
\qquad\forall\, (T,x)\in\RR_+\times\Rd\,,
\end{equation*}
for any $v\in\Usm^*$, and, in addition, \cref{E-crucial} holds.

\smallskip
\item[\ttup d]
The function $\psi=\log\Psi$ satisfies $\abs{\grad\psi} \le \widetilde{C}_0
(1+\abs{x})$ for some constant $\widetilde{C}_0$.

\smallskip
\item[\ttup e]
If $\lambda_*$ is strictly monotone at $c$ on the right,
then there exists a unique (up to a positive multiplicative constant) positive
solution to \cref{T3.1A} (ground state),
and any optimal $v\in\Usm$ lies in $\Usm^*$.
In addition, the ground state $\Psi$ satisfies \cref{E-strep},
and \cref{E-sde*} controlled under $U_t=v^*(X^*_t)$ with $v^*\in\Usm^*$
is recurrent.
\end{enumerate}
\end{theorem}

There is another important property that we need in the study of
convergence of the value iteration, which we explain next.
Let $v\in\Usm$.
We say that $\lambda_v(c)$, defined in
\cref{E-lamstrv}, is \emph{strictly monotone at $c$} if
$\lambda_v(c-h)<\lambda_v(c)$ for some $h\in C_{\mathrm{o}}^+(\Rd)$.
Of course, strict monotonicity implies strict monotonicity on the right
as can be seen from the fact that $c\mapsto\lambda_v(c)$ is convex.
By \cite[Theorem~2.1]{ABS-19} strict monotonicity of
$\lambda_{v^*}(c)$ at $c$ is equivalent to the statement
that the ground state diffusion in \cref{E-sde*} controlled under $v^*$ is
positive recurrent.

\subsection{The value iteration}

Let
\begin{equation*}
C^2_{\Psi,+}(\Rd)\,\df\, \bigl\{g\in C^2(\Rd)\,\colon\,
g>0\,,~ \norm{g}_{\mathstrut\Psi} <\infty\bigr\}\,.
\end{equation*}
We introduce the equation
\begin{equation}\label{E-VI}
\partial_t\,\overline\Phi(t,x)  \,=\, \min_{u\in\Act}\;
\bigl[\Lg_u \overline\Phi(t,x) + c(x,u)\,\overline\Phi(t,x)\bigr]
- \lambda_*\,\overline\Phi(t,x)\,,\quad t>0\,,
\end{equation}
with $\overline\Phi(0,x)=\Phi_0(x)$, $\Phi_0\in C^2_{\Psi,+}(\Rd)$.

\begin{definition}\label{D3.1}
Let $\{\Hat{v}_t\}_{t\ge0}$ be an a.e.\ measurable selector from
the minimizer of \cref{E-VI}.
We define
the corresponding (nonstationary) Markov control
\begin{equation*}
\Hat{v}^{t}\df\bigl\{\Hat{v}^{t}_{s}\,=\,\Hat{v}_{t-s}(x)\,,\; s\in[0,t]\bigr\}\,.
\end{equation*}
and denote the set of these controls by $\widehat\cU(\Phi_0)$, including
explicitly the dependence on the initial condition $\Phi_0$ in the notation.
\end{definition}

We don't care so much about uniqueness of solutions to \cref{E-VI};
however, see \cite[Theorems~3.12--3.13]{RVIM}.
We work with the solution $\overline\Phi(t,x)$ which satisfies
\begin{equation*}
\begin{aligned}
\overline\Phi(t,x) &\,=\, \inf_{U\in\Uadm}\,
\Exp^{\mathstrut x}_U \Bigl[\E^{\int_{0}^{t}
[c(X_{s},U_s)-\lambda_*]\,\D{s}}\, \Phi_0(X_{t})\Bigr]\\
&\,=\,
\Exp^{x}_{\Hat{v}^{t}}\Bigl[\E^{\int_{0}^{t}
[c(X_{s},\Hat{v}^{t}_s(X_s))-\lambda_*]\,\D{s}}\, \Phi_0(X_{t})\Bigr]
\qquad\forall\, \{\Hat{v}^{t}\}_{t\ge0}\in \widehat\cU(\Phi_0)\,.
\end{aligned}
\end{equation*}
Note that for any element of  $\widehat\cU(\Phi_0)$ we have
$\Hat{v}^{t+\tau}_{s+\tau}=\Hat{v}^{t}_{s}$
for all $t\ge s\ge0$ and $\tau\ge0$.
Also, by \cref{E-crucial}, we obtain
\begin{equation*}
\Psi(x) \,\le\, \Exp^{x}_{\Hat{v}^{t}}\Bigl[\E^{\int_{0}^{t}
[c(X_{s},\Hat{v}^{t}_s(X_s))-\lambda_*]\,\D{s}}\, \Psi(X_{t})\Bigr]
\qquad\forall\, \{\Hat{v}^{t}\}_{t\ge0}\in \widehat\cU(\Phi_0)\,.
\end{equation*}

Incorporating explicitly the dependence
on the initial condition
$\Phi_0$ in the notation,
we let $\cS_t[\Phi_0](x)$, $t\ge0$, denote the solution of \cref{E-VI}.
It is clear that $\cS_t[\Psi]=\Psi$ for all $t\ge0$ by \cref{T3.1}\,(c),
and that the uniqueness of the ground state in \cref{T3.1}\,(e) implies
that any positive initial condition $\Phi_0$ satisfying
$\cS_t[\Phi_0]=\Phi_0$ for all $t\ge0$ must equal the ground state $\Psi$
up to a positive multiplicative constant.

Let $\Equil$ denote the set of equilibria of the semiflow $\cS_t$,
or equivalently,
the set of solutions of the HJB in \cref{T3.1A}, that is,
\begin{equation*}
\Equil \,\df\, \{r\Psi\colon\, r>0\}\,.
\end{equation*}
By $C_\Psi(\Rd)$ we denote the class of continuous functions
$\phi$ satisfying
\begin{equation*}
\norm{\phi}_{\mathstrut\Psi} \,\df\, \sup_{x\in\Rd}\,\frac{\abs{\phi(x)}}{\Psi(x)}
\,<\,\infty\,.
\end{equation*}
For $\kappa>0$ we define the set $\cH_{\kappa}\subset C^2(\RR^{d})$ by
\begin{equation}\label{E-cH}
\cH_{\kappa} \df\bigl\{h\in C^2(\RR^{d})\,\colon\,
h\,\ge\, \kappa^{-1}\Psi\,,~ \norm{h}_{\mathstrut\Psi}<\kappa\bigr\}\,.
\end{equation}
We have
\begin{equation}\label{E-inv}
\begin{aligned}
\kappa^{-1}\Psi(x) &\,=\,  \cS_t[\kappa^{-1}\Psi](x)\\
&\,\le\, \cS_t[\Phi_0](x)\\
&\,\le\, \cS_t\bigl[\norm{\Phi_0}_{\mathstrut\Psi}\Psi\bigr](x)\\
&\,\le\, \cS_t\bigl[\kappa\Psi\bigr](x)
\,=\, \kappa \Psi(x)
\qquad\forall\,\Phi_0\in\cH_\kappa\,,
\end{aligned}
\end{equation}
where the first and the last equalities follow by \cref{T3.1}\,(c), and the
inequalities by the monotonicity of $f\mapsto \cS_t[f]$ and
the definition of $\cH_{\kappa}$.
It follows from \cref{E-inv} that if $\Phi_0\in\cH_\kappa$ then
$\cS_t[\Phi_0]\in\cH_\kappa$ for all $t\ge0$.
So the set $\cH_\kappa$ is positively invariant under the semiflow $\cS_t$.

Recall the definition of $\Lg$ in \cref{E-Lg}, and let
\begin{equation}\label{E-TLg}
\widetilde\Lg_u \,\df\,  \Lg_u +
\bigl\langle \grad \psi(x), a(x)\grad\bigr\rangle\,,\qquad u\in\Act\,.
\end{equation}
This definition can be extended to $\widetilde\Lg_v$ for any Markov control
$v$ (not necessarily stationary) by replacing $u\in\Act$ with $v$
in \cref{E-TLg}.
Clearly then $\Lg_v$, with $v\in\Usm$, is the extended generator
of \cref{E-sde*} controlled by $v$.
The operator $\widetilde\Lg_u$ satisfies a very important identity.
If $\Phi\in C^2(\Rd)$ is a positive function then
\begin{equation}\label{E-identity}
\widetilde\Lg_{u}\Bigl(\frac{\Phi}{\Psi}\Bigr)
\,=\, \biggl(\frac{\Lg_{u} \Phi}{\Phi} -\frac{\Lg_{u} \Psi}{\Psi}\biggr)
\frac{\Phi}{\Psi}
\qquad\forall\,u\in\Act\,.
\end{equation}

In the sequel we work under the following hypothesis.

\begin{itemize}
\item[\hypertarget{H1}{\sf{(H1)}}]
{\slshape The ground state diffusion in \cref{E-sde*} is positive recurrent
under some $v^*\in\Usm^*$.
We let $\Tilde\mu_*$ denote its invariant probability measure, and
$\widetilde\Exp^x_*$ expectation operator on the canonical space of the process
controlled under $v^*$.}
\end{itemize}

As explained in \cref{S3.1}, under \hyperlink{H1}{\sf{(H1)}},
$\lambda_{v^*}(c)$ is strictly monotone at $c$.
Therefore, by \cref{T3.1}, we have unicity of the ground state $\Psi$,
and complete verification of optimality results.
In what follows $v^*$ is the control in \hyperlink{H1}{\sf{(H1)}}.

We present the following important convergence result.

\begin{theorem}\label{Tloc}
Grant \hyperlink{H1}{\sf{(H1)}}.
For each $\Phi_0\in\cH_{\kappa}$, $\kappa>0$, the semiflow
$\cS_t[\Phi_0]$ converges to
$\kappa_{0}\Psi\in\Equil$ for some $\kappa_0\in[\kappa^{-1},\kappa]$ as $t\to\infty$.
Moreover, if $A$ is a bounded subset of $C_\Psi(\RR^{d})$, then the only subsets
of $\cH_\kappa\cap A$, with $\kappa>0$, which are invariant under the
semiflow
are the points (singletons) of $\Equil\cap \cH_\kappa\cap A$.
\end{theorem}

\begin{proof}
Define $\Phi_{\mathstrut\Psi}(t,x)\df \frac{\cS_t[\Phi_0](x)}{\Psi(x)}$.
By \cref{E-identity} applied to \cref{E-VI,E-eigen}, we have
\begin{equation}\label{ETlocA}
\partial_{t} \Phi_{\mathstrut\Psi}(t,x) - \widetilde\Lg_{v^*} \Phi_{\mathstrut\Psi}(t,x)
\,\le\, 0\,.
\end{equation}
Since $\Phi_{\mathstrut\Psi}(t,x)$ is bounded by \cref{E-inv},
we obtain from \cref{ETlocA} that
\begin{equation}\label{ETlocB}
\Phi_{\mathstrut\Psi}(t,x)\,\le\,
\widetilde\Exp^x_*\bigl[\Phi_{\mathstrut\Psi}(\tau,X_\tau)\bigr]\,,
\qquad 0\le \tau \le t\,.
\end{equation}
Integrating \cref{ETlocB} with respect to $\Tilde\mu_*$,
and using the abbreviated notation $\Tilde\mu_*(f)=\int_\Rd f(x)\,\Tilde\mu_*(\D{x})$,
we obtain
\begin{equation*}
\Tilde\mu_*\bigl(\Phi_{\mathstrut\Psi}(t,x)\bigr) \,\le\,
\Tilde\mu_*\bigl(\Phi_{\mathstrut\Psi}(s,x)\bigr)\quad\text{for all\ }t>s\,.
\end{equation*}
Thus, since $t\mapsto\Tilde\mu_*\bigl(\Phi_{\mathstrut\Psi}(t,x)\bigr)$ is nonincreasing,
and $\Phi_{\mathstrut\Psi}(t,x)\in\cH_{\kappa}$ by \cref{E-inv},
it converges to some constant $\kappa_0\in[\kappa^{-1},\kappa]$
as $t\to\infty$.
It is clear that
$\sup_{t>0}\;\norm{\cS_t[\Phi_0]}_{\mathstrut\Psi}
<\norm{\Phi_0}_{\mathstrut\Psi}$ by \cref{E-inv}.
Therefore by the interior estimates of solutions of \cref{E-VI}
(see \cite[Theorem~6.2, p.~457]{Lady}),
$\bigl\{\cS_t[\Phi_0]\,,\;t>0\bigr\}$ is locally
precompact in $C^2(\RR^{d})$.
Hence the $\omega$-limit set of $\Phi_0$ under the semiflow $\cS_t$,
denoted by $\omega(\Phi_0)$, is nonempty, and is a subset of $C^2(\Rd)$.
Note that the convergence of $\Tilde\mu_*\bigl(\Phi_{\mathstrut\Psi}(t,x)\bigr)$
to $\kappa_0$ as $t\to\infty$ implies that
\begin{equation}\label{ETlocC}
\Tilde\mu_*\Bigl(\frac{h}{\Psi}\Bigr)\,=\,\kappa_0\quad
\forall\, h\in\omega(\Phi_0)\,.
\end{equation}

Fix some $h\in\omega(\Phi_0)$,
and define
\begin{equation}\label{ETlocD}
g(t,x)\,\df\, \Lg_{v^*} \cS_t[h](x) + c_{v^*}(x)\,\cS_t[h](x)
- \min_{u\in\Act}\;
\bigl[\Lg_{u} \cS_t[h](x) + c(x,u)\,\cS_t[h](x)\bigr]\,.
\end{equation}
Therefore, by \cref{E-VI,ETlocD}, we have
\begin{equation}\label{ETlocE}
\partial_t\,\overline\Phi(t,x)  \,=\,
\Lg_{v^*} \cS_t[h](x) + c_{v^*}(x)\,\cS_t[h](x)
-g(t,x)
- \lambda_*\,\overline\Phi(t,x)\,,\quad t>0\,,
\end{equation}
which we write as
\begin{equation*}
\partial_t\,\overline\Phi(t,x)  \,=\,
\Lg_{v^*} \cS_t[h](x) +
\biggl(c_{v^*}(x)-\frac{g(t,x)}{\cS_t[h](x)}\biggr)\,\cS_t[h](x)
- \lambda_*\,\overline\Phi(t,x)\,.
\end{equation*}
Using \cref{E-VI,E-identity,ETlocE}, we obtain
\begin{equation}\label{ETlocF}
\partial_{t} \frac{\cS_t[h](x)}{\Psi(x)} - \widetilde\Lg_{v^*} \frac{\cS_t[h](x)}{\Psi(x)}
\,=\, -\frac{g(t,x)}{\Psi(x)}\,.
\end{equation}
Since $\norm{\cS_t[h]}_{\mathstrut\Psi}\le\kappa$ by the positive
invariance of $\cH_\kappa$,
we can apply It\^o's formula to \cref{ETlocF} to obtain
\begin{equation}\label{ETlocG}
\frac{\cS_t[h](x)}{\Psi(x)}\,=\, -\widetilde\Exp^x_*\biggl[
\int_{0}^{t}\frac{g(t-s,X_s)}{\Psi(X_s)}\,\D{s}\biggr]
+ \widetilde\Exp^x_*\biggl[\frac{h(X_t)}{\Psi(X_t)}\biggr]
\qquad \forall\, t>0\,.
\end{equation}
As argued earlier $t\mapsto\Tilde\mu_*\Bigl(\frac{\cS_t[h](x)}{\Psi(x)}\Bigr)$
is constant.
Hence, integrating
\cref{ETlocG} with respect to $\Tilde\mu_*$,
we obtain
\begin{equation*}
\int_{0}^{t}\int_{\RR^{d}} g(t-s,x)\,\frac{1}{\Psi(x)}\,
\Tilde\mu_*(\D{x})\,\D{s}\,=\,0
\quad\Longrightarrow\quad g(t,x)=0\quad (t,x)-\text{a.e.}
\end{equation*}
where we used the fact that  $\Psi(x)>0$.
Therefore, the first term on the right-hand side of \cref{ETlocD} is
identically equal to $0$.
Since $\frac{h}{\Psi}$ is bounded and the diffusion governed by $\Tilde\Lg^*$
is ergodic, the second term on the right hand side of \cref{ETlocG}
converges as $t\to\infty$ to some constant $\kappa_0$ by \cref{ETlocC}.
Thus, again by \cref{ETlocG},
 $\cS_t[h]$ converges to $\kappa_0 \Psi$ along any subsequence
as $t\to\infty$, and the invariance of
the $\omega$-limit set of $\cS_t[\Phi_0]$ implies that
$h = \kappa_0 \Psi$.
This completes the proof.
\end{proof}

\subsection{The relative value iteration}

We modify \cref{E-VI} as follows:
\begin{equation}\label{E-RVI}
\partial_t\,\Phi(t,x)  \,=\, \min_{u\in\Act}\;
\bigl[\Lg_{u} \Phi(t,x) + f(x,u)\,\Phi(t,x)\bigr] - \Phi(t,0)\,\Phi(t,x)\,,
\quad t>0\,,
\end{equation}
with $\Phi(0,x)=\Phi_0(x)$.
Existence of solutions to \cref{E-RVI} is evident from the following
observation:
If $\Phi$ solves \cref{E-RVI} then
\begin{equation}\label{RVI01}
\overline\Phi(t,x) \,=\, \Phi(t,x)\,
\E^{\int_{0}^{t}(\Phi(s,0)-\lambda_*)\,\D{s}}
\end{equation}
solves \cref{E-VI}.
Therefore,
\begin{equation}\label{RVI02}
\frac{\overline\Phi(t,x)}{\Phi(t,x)} \,=\,
\frac{\overline\Phi(t,0)}{\Phi(t,0)}
\qquad \forall\,(t,x)\in(0,\infty)\times\Rd\,,
\end{equation}
so that $\frac{\overline\Phi(t,x)}{\Phi(t,x)}$ does not depend on $x$.
By \cref{RVI01}--\cref{RVI02} we have
\begin{equation*}
\begin{aligned}
\frac{\D}{\D{t}}\,\frac{\Phi(t,x)}{\overline\Phi(t,x)}
&\,=\,-\Phi(t,0)+\lambda_*\nonumber\\[5pt]
&\,=\,-\overline\Phi(t,0)\,\frac{\Phi(t,x)}{\overline\Phi(t,x)}
+\lambda_*\,.
\end{aligned}
\end{equation*}
Thus
\begin{equation}\label{RVI04}
\frac{\Phi(t,x)}{\overline\Phi(t,x)}\,=\,
\E^{-\int_{0}^{t}\overline\Phi(s,0)\,\D{s}}
+ \lambda_* \int_{0}^{t} \E^{-\int_{\tau}^{t}\overline\Phi(s,0)\,\D{s}}\,\D\tau\,.
\end{equation}
It follows by \cref{RVI04} that if
$\overline\Phi(t,0)\to C>0$ as $t\to\infty$ for some positive constant $C$,
then
$\frac{\Phi(t,x)}{\overline\Phi(t,x)}$ converges to a positive constant
as $t\to\infty$.
and thus by \cref{RVI01} we have
\begin{equation*}
\int_{0}^{t}(\Phi(s,0)-\lambda_*)\,\D{s}\;\xrightarrow[t\to\infty]{}\;
\text{constant.}
\end{equation*}
In particular $\Phi(t,0)\to\lambda_*$ as $t\to\infty$.

\subsection{Results under blanket exponential ergodicity}

Under blanket exponential ergodicity, we can remove
the hypotheses in \cref{A3.1}\,(ii).
We keep \cref{A3.1}\,(i), and add an affine growth condition of the form
\begin{equation}\label{E-affine}
\sup_{u\in\Act}\; \langle b(x,u),x\rangle^{+} + \norm{\upsigma(x)}^{2}\,\le\,C_0
\bigl(1 + \abs{x}^{2}\bigr) \qquad \forall\, x\in\RR^{d}\,.
\end{equation}
Concerning the running cost, we assume that it is bounded below,
and, without loss of generality, we normalize it so
that $\inf_{\Rd\times\Act}\,c=0$.

The essential hypothesis in this subsection is the following.

\begin{assumption}\label{ALyap}
We distinguish two cases.
\begin{enumerate}
\item[(i)]
If $c$ is bounded, we assume that there exist
a function $\Lyap \in C^2(\Rd)$ taking
values in $[1,\infty)$, a compact set $\cK\subset\Rd$,
and constants $\widehat{C}$ and $\gamma>\norm{c}_\infty$
which satisfy
\begin{equation}\label{ALyapA}
\Lg_u \Lyap(x)
\,\le\, \widehat{C} \Ind_{\cK}(x)-\gamma \Lyap(x)\qquad\forall\,u\in\Act\,.
\end{equation}

\item[(ii)] If $c$ is not bounded, we assume
that there exist an inf-compact function $F$ and a constant $\beta\in(0,1)$
such that $\beta F-c$ is also inf-compact, and $\Lyap$, $\cK$, and $\widehat{C}$
as in part (i), such that
\begin{equation}\label{ALyapB}
\Lg_u \Lyap(x)
\,\le\, \widehat{C} \Ind_{\cK}(x)-F(x) \Lyap(x)\,.
\end{equation}
\end{enumerate}
\end{assumption}

The reason for differentiating cases (i) and (ii)
in \cref{ALyap} is because if the coefficients $a$ and $b$ are
bounded, it is not, in general, possible to find an inf-compact function $F$
which satisfies \cref{ALyapB}.

Under \cref{ALyap} we obtain a must stronger version of \cref{T3.1}.
Recall the definitions in \cref{E-valueU,E-value,E-lamstr},
and $\Usm^*$ in the beginning
of \cref{S3.1}.
The following theorem is a combination of \cite[Theorems~4.1 and 4.2]{ABS-19},
and the results in \cite[Section~3]{ABS-19}.

\begin{theorem}\label{TLyap}
Grant \cref{A3.1}\,\ttup{i}, \cref{E-affine},
and \cref{ALyap}.
Then $\lambda_*$ is finite, and
the equation
\begin{equation}\label{ETLyapA}
\min_{u\in\Act}\;
\bigl[\Lg_u \Psi(x) + c(x,u)\,\Psi(x)\bigr] \,=\,  \lambda_*\,\Psi(x)
\qquad\forall\,x\in\Rd
\end{equation}
has a unique positive solution $\Psi\in C^2(\RR^{d})$,
$\Psi(0)=1$, and the following hold.
\begin{enumerate}
\item[\ttup a]
$\Lambda^x_*=\Lambda_*=\lambda_*$ for all $x\in\Rd$.

\smallskip
\item[\ttup b]
A stationary Markov control is optimal, if and only if it belongs to $\Usm^*$.

\smallskip
\item[\ttup c]
Part \ttup{c} of \cref{T3.1} holds, and also \cref{E-strep}.

\smallskip
\item[\ttup d]
The ground state diffusion \cref{E-sde*} is exponentially ergodic
under any stationary Markov control.
\end{enumerate}
\end{theorem}

We review part (d) of \cref{TLyap} which is not discussed in \cite{ABS-19}.
First, it is straightforward to show, by using \cref{ALyapA,ALyapB} as a
barrier in the construction of the solution $\Psi$, that
$\frac{\Lyap}{\Psi}$ is bounded away from $0$ on $\Rd$.
Second, note that the nonnegativity of $c$ implies
that $\lambda_*\ge0$,
Thus, from \cref{E-identity,ALyapA,ETLyapA} we obtain
\begin{equation}\label{ETLyapB}
\widetilde\Lg_{u}\Bigl(\frac{\Lyap}{\Psi}\Bigr)(x)\,\le\,
\Bigl(\Lyap^{-1}(x)\widehat{C} \Ind_{\cK}(x)-\lambda_*+
c(x\,,u)-\gamma\Bigr)\frac{\Lyap(x)}{\Psi(x)}
\qquad\forall\,(x,u)\in\Rd\times\Act\,.
\end{equation}
Under \cref{ALyapB},
$\gamma$ gets replaced by $F$ in \cref{ETLyapB}.
It is well known (see \cite[Lemma~2.5.5]{book}) that \cref{ETLyapB} implies that
there exist positive constants $\Tilde{\kappa}_0$ and $\Tilde{\kappa}_1$ such
that
\begin{equation}\label{E-geom1}
\widetilde\Exp^{\mathstrut x}_U \biggl[\frac{\Lyap}{\Psi}(X_{t}) \biggr] \,\le\,
\Tilde{\kappa}_0 + \frac{\Lyap}{\Psi}(x)\, \E^{-\Tilde{\kappa}_1 t}
\qquad \forall\, x\in\RR^{d}\,,\ \forall\, U\in\Uadm\,.
\end{equation}
Let $\widetilde{P}^v_t(x,\D{y})$ denote the transition
probability of the process $\process{X^*}$ in \cref{E-sde*} under
the control $v\in\Usm$, and $\Tilde\mu_v$ its invariant probability measure.
Then, using the argument as in the proof
of \cite[Theorem~2.1\,(b)]{AHP18}, one can show
that \cref{ETLyapB} implies that there exist positive constants
$\gamma_\circ$ and $C_{\gamma_\circ}$, which do not depend on $v\in\Usm$, such that
\begin{equation*}
\bnorm{\widetilde{P}^v_t(x,\cdot\,)-\Tilde\mu_v(\cdot)\,}_{\tv}\,\le\,
C_{\gamma_\circ} \frac{\Lyap(x)}{\Psi(x)}\, \E^{-\gamma_\circ t}\qquad
\forall\, (t,x)\in\RR_+\times\Rd\,,
\end{equation*}
where $\norm{\,\cdot\,}_\tv$ denotes the total variation norm.

\begin{remark}
We want to point out that the proof of \cite[Theorems~4.1 and 4.2]{ABS-19},
shows that under the hypotheses of \cref{TLyap},
the generalized principal eigenvalue $\lambda_v$ defined in \cref{E-lamstrv}
is finite for any $v\in\Usm$, and
there exists a positive $\Psi_v\in\Sobl^{2,p}(\Rd)$, for any $p\ge d$,
which solves
\begin{equation}\label{ERLyapA}
\Lg_v \Psi_v(x) + c_v(x)\,\Psi_v(x) \,=\,  \lambda_v\,\Psi_v(x)
\qquad\text{a.e.\ }\,x\in\Rd\,.
\end{equation}
In addition, $\Psi_v$ is the unique positive solution of \cref{ERLyapA}
in $\Sobl^{2,d}(\Rd)$
up to a positive multiplicative constant, and $\lambda_v=\Lambda^x_v$ for all $x\in\Rd$,
or in other words, the risk-sensitive value equals the generalized principal
eigenvalue of the operator $\Lg_v+ c_v$.
Another important result is given in \cite[Theorem~4.3]{ABS-19} which
shows that, under \cref{ALyap}, $v\mapsto \lambda_v$ is continuous
in the topology of Markov controls (see \cite{Borkar-89} for a definition
of this topology).
\end{remark}

Moving on to the VI algorithm under the assumptions of \cref{TLyap}, note
that by \cref{ETlocB} we have
\begin{equation}\label{EstableA}
\Phi_{\mathstrut\Psi}(t,x)\,=\,\frac{\cS_t[\Phi_0](x)}{\Psi(x)}
\,\le\, \widetilde\Exp^x_*\biggl[\frac{\Phi_0}{\Psi}(X_t)\biggr]
\qquad \forall\, t\ge0\,.
\end{equation}
This gives us an upper bound.
To obtain a lower bound, we use the measurable selector
$\{\Hat{v}_t\}$ in \cref{D3.1} and combine
 \cref{E-VI,E-identity,ETLyapA}, to write
\begin{equation}\label{EstableB}
\partial_{t} \Phi_{\mathstrut\Psi}(t,x) - \widetilde\Lg_{\Hat{v}_t} \Phi_{\mathstrut\Psi}(t,x)
\,\ge\, 0\,.
\end{equation}
With $\Bar\varphi(t,x)\df\log \overline\Phi(t,x)$
and $\varphi_0\df\log \Phi_0$, we deduce from
\cref{EstableB} that
\begin{equation}\label{EstableC}
\Bar\varphi(t,x) \,\ge\, \psi(x) +
\widetilde\Exp^x_{\Hat{v}^{t}}\bigl[\varphi_0(X_t) -\psi(X_t)\bigr]\,,
\end{equation}
where the expectation is under the nonstationary control
$\{\Hat{v}^{t}\}_{t\ge0}\in \widehat\cU(\Phi_0)$ in \cref{D3.1}.

We borrow the following result.
As shown in the proof of \cite[Theorem~4.3]{ABS-19},
under \cref{ALyap}, there exist
positive constants $\Hat\kappa_0$, and $\delta_\circ>1$
such that
$\Lyap \ge \Hat\kappa_0 \Psi^{\,\delta_\circ}$.
This together with \cref{E-geom1} and Jensen's inequality shows that
there exists a constant $\Hat\kappa_1$ such that
\begin{equation}\label{EstableD}
\widetilde\Exp^x_{\Hat{v}^{t}}\bigl[\psi(X_t)\bigr] \,\ge\,
\frac{1}{\delta_\circ-1}\,\log\biggl(
\frac{\Tilde{\kappa}_0}{\Hat\kappa_0}
+ \frac{\Lyap(x)}{\Hat\kappa_0\Psi(x)}\, \E^{-\Tilde{\kappa}_1 t}\biggr)\,.
\end{equation}
Combining \cref{EstableC,EstableD}, we obtain
\begin{equation}\label{EstableE}
\liminf_{t\to\infty}\,\Bar\varphi(t,x) \,\ge\, \psi(x)
+\biggl(\inf_\Rd\,\varphi_0\biggr)
+ \frac{1}{\delta_\circ-1}\,\log\biggl(
\frac{\Tilde{\kappa}_0}{\Hat\kappa_0}\biggr)\,.
\end{equation}
\Cref{EstableA,EstableE} shows that as long as the initial condition $\Phi_0$ is bounded
from below
away from $0$ in $\Rd$, and $\norm{\Phi_0}_\Lyap<\infty$, then any limit point
in $C^2(\Rd)$
of the semiflow $\cS_t[\Phi_0]$ lies in the set $\cH_\kappa$ for some
$\kappa>0$ (recall the definition in \cref{E-cH}).
Using the interior estimates of solutions
and the bounds in \cref{EstableA,EstableC,EstableD}, as in the proof of \cref{Tloc},
it is straightforward to show that the $\omega$-limit set of $\Phi_0$
is a non-empty subset of $C^2(\Rd)$, therefore
also of $\cH_\kappa$.
Hence, following the arguments in \cite[Section~4.2]{RVI} which is
based on convergence of reverse supermartingales, or
the method in \cite{ABor-17}  that has a dynamical systems flavor
(see also \cite[Theorem~3.1]{ISDG-13}),
one can establish the following result.

\begin{theorem}\label{T-stable}
Grant \cref{A3.1}\,\ttup{i}, \cref{E-affine},
and \cref{ALyap}, and suppose that the initial condition
$\Phi_0\in C^2(\Rd)$ is bounded from below away from $0$,
and satisfies $\norm{\Phi_0}_\Lyap<\infty$.
Then there exists a positive constant $\kappa_0=\kappa_0(\Phi_0)$ such that
the value iteration $\overline\Phi(t,x)$ in \cref{E-VI} converges
to $\kappa_0\Psi(x)$ as $t\to\infty$ uniformly on compact sets.
\end{theorem}

\begin{remark}
When the state space is compact, stronger results can be obtained.
Such a scenario is investigated in \cite{Collatz}, and
Theorem~4.3 in that paper shows in fact that under mild assumptions,
and for a large class of abstract problems,
the convergence is exponential.
\end{remark}

\begin{remark}
It is worth investigating if the global convergence result in \cref{T-stable}
holds under additional assumptions in the near-monotone case.
Suppose that $\theta=1$ in \cref{EA3.1B,EA3.1D} and that $c$ has strictly
quadratic growth.
Then, by \cref{E-gradient}, $c$ satisfies
\begin{equation}\label{ER3.3}
\min_{u\in\Act}\,c(x,u) \,\ge\,
\theta_{1} \psi(x) - \theta_{2}
\qquad\forall\, x\in\RR^{d}\,.
\end{equation}
for some positive constants $\theta_1$ and $\theta_2$.
In the case of the ergodic control problem,
under the structural condition in \cref{ER3.3}, with
$\psi$ replaced by the solution of the HJB equation, global convergence
can be established for the value iteration in continuous
\cite[Theorem~3.2]{RVIM}, as well as in discrete time
\cite[Theorems~6.1--6.2]{ABor-dada} (see also \cite{Ari-open}).
For the risk-sensitive problem, this inequality has to be modfied to account
for the \emph{relative entropy rate} term arising from the logarithmic transformation.
We strengthen \cref{ER3.3} to
\begin{equation}\label{ER3.4}
\min_{u\in\Act}\,c(x,u) -\frac{1}{2}\babs{\upsigma\transp(x)\grad\psi(x)}^2\,\ge\,
\theta_{1} \psi(x) - \theta_{2}
\qquad\forall\, x\in\RR^{d}\,.
\end{equation}
Note that \cref{ER3.4} implies \hyperlink{H1}{\sf{(H1)}}.
We conjecture that under the structural assumption in \cref{ER3.4}
the value iteration $\overline\Phi(t,x)$ in \cref{E-VI},
starting from any initial condition $\Phi_0\in C^2_{\Psi,+}(\Rd)$,
converges to an equilibrium in $\Equil$.
\end{remark}

\section*{Acknowledgments}
The work of Ari Arapostathis was supported in part by 
the Army Research Office through grant W911NF-17-1-001,
in part by the National Science Foundation through grant DMS-1715210,
and in part by the Office of Naval Research through grant N00014-16-1-2956
and was approved for public release under DCN \#43-6054-19.
The work of Vivek Borkar was supported by a J.\ C.\ Bose Fellowship.


\end{document}